\documentclass{amsart}

\usepackage{amsfonts,amssymb,amsmath,enumerate,verbatim,mathtools,tikz,bm,mathrsfs,tikz-cd,hyperref,comment,stmaryrd,amsthm}
\usepackage{colonequals}
\usepackage{adjustbox}
\usepackage[all,pdf]{xy}
\hypersetup{
    colorlinks=true,
    linkcolor=blue,
    filecolor=magenta,      
    urlcolor=cyan,
}
\usepackage{cleveref}

\crefname{lemma}{lemma}{lemmas}
\Crefname{lemma}{Lemma}{Lemmas}
\crefname{proposition}{proposition}{propositions}
\Crefname{proposition}{Proposition}{Propositions}
\crefname{corollary}{corollary}{corollaries}
\Crefname{corollary}{Corollary}{Corollaries}
\crefname{remark}{remark}{remarks}
\Crefname{remark}{Remark}{Remarks}

\usepackage{adjustbox}
\usepackage{leftindex}

\makeatletter
\@namedef{subjclassname@2020}{\textup{2020} Mathematics Subject Classification}
\makeatother

\author[Jian Liu]{Jian Liu}
\address{School of Mathematics and Statistics, and Hubei Key Laboratory of Mathematical Sciences,  Central China Normal University,  Wuhan 430079, P.R. China}
\email{jianliu@ccnu.edu.cn}

\keywords{dominant local ring, descent, singularity category, homotopy category, acyclic complex, injective module}
\subjclass[2020]{18G80 (primary); 13C11, 13D09 (secondary)}

\DeclareMathOperator{\rank}{rank}

\DeclareMathOperator{\embdim}{embdim}

\DeclareMathOperator{\h}{H}

\newcommand{\n}{\mathfrak{n}}

\newcommand{\Z}{\mathbb{Z}}

\newcommand{\T}{\mathsf{T}}

\newcommand{\D}{\mathsf{D}}
\newcommand{\K}{\mathsf{K}}
\newcommand{\C}{\mathsf{C}}

\pgfdeclarelayer{bg}    
\pgfsetlayers{bg,main}

\newcommand{\del}{\partial}

\newcommand{\m}{\mathfrak{m}}

\DeclareMathOperator{\pd}{pd}

\DeclareMathOperator{\ac}{ac}

\DeclareMathOperator{\Hom}{Hom}

\DeclareMathOperator{\Ext}{Ext}

\DeclareMathOperator{\thick}{\mathsf{thick}}

\DeclareMathOperator{\Tor}{Tor}

\DeclareMathOperator{\codim}{codim}

\newcommand{\per}{\mathsf{perf}}
\newcommand{\sg}{\mathsf{sg}}

\newcommand{\Mod}{\mathsf{Mod}}
\newcommand{\Inj}{\mathsf{Inj}}
\newcommand{\mo}{\mathsf{mod}}

\newtheorem{theorem}{Theorem}[section]

\newtheorem*{Question}{Question}

\newtheorem{proposition}[theorem]{Proposition}

\newtheorem{lemma}[theorem]{Lemma}
\newtheorem{corollary}[theorem]{Corollary}

\theoremstyle{definition}
\newtheorem{example}[theorem]{Example}
\newtheorem{remark}[theorem]{Remark}

\newtheorem{chunk}[theorem]{}

\usepackage[utf8]{inputenc}

\newtheorem*{ack}{Acknowledgements}

\title[On Takahashi's descent question]{On Takahashi's descent question about dominant local rings}
\date{\today}

\begin{document}

\begin{abstract}
This article studies dominant local rings, a notion introduced by Takahashi. We prove that, for a flat local homomorphism $R\longrightarrow S$ between commutative noetherian local rings, dominance descends from $S$ to $R$, thereby answering Takahashi's descent question affirmatively. The proof makes use of Krause's realization of the idempotent completion
of the singularity category as the subcategory of compact objects in the
homotopy category of acyclic complexes of injective modules.
\end{abstract}
\maketitle

\section{Introduction}

Let $R$ be a commutative noetherian local ring. In \cite{Takahashi:2023}, Takahashi introduced the notion of \emph{dominant local rings}. By definition, $R$ is dominant if the residue field of $R$ belongs to every nonzero thick subcategory of the singularity category of $R$; see \Cref{dominant}. Dominant local rings enjoy several desirable homological properties, including being $\Tor$-friendly.  In \cite{Takahashi:2023}, Takahashi verified that several important classes of local rings are dominant. These include hypersurfaces, Cohen--Macaulay local rings with infinite residue fields and minimal multiplicity, local rings with quasi-decomposable maximal ideals introduced in \cite{NT}, and Burch rings introduced in \cite{DKT2020}. 

The classification of thick subcategories is a fundamental topic in algebra and representation theory. The classification of thick subcategories of the category of perfect complexes over a commutative noetherian ring dates back to the work of Hopkins \cite{Hopkins} and Neeman \cite{Neeman92-DR}. In \cite{Takahashi:2010}, Takahashi classified the thick subcategories of the singularity category of a hypersurface. A generalization of this result to complete intersections was established by Stevenson \cite{Stevenson}.  Recently, in \cite{Takahashi:2023}, Takahashi classified the thick subcategories of the singularity category of a local ring whose certain localizations are dominant local rings. This result highlights the importance of dominant local rings. 

In \cite{Takahashi:2026}, Takahashi introduced and studied the notion of uniformly dominant local rings. Kobayashi and Takahashi \cite{KT2026} further developed this theory. More recently, Kimura, Mifune, Otake, and Takahashi \cite{KMOT} studied the notion of quasi-dominant local rings introduced in \cite{Takahashi:2023} and showed that it is closely related to local rings over which every Gorenstein projective module is projective.

Flat local homomorphisms are among the most important morphisms in commutative algebra. Many important properties of local rings are known to descend along flat local homomorphisms. In \cite[Question 7.5]{Takahashi:2023}, Takahashi raised the following natural question:

\begin{Question}
    Let $R\longrightarrow S$ be a flat local homomorphism between commutative noetherian local rings. If $S$ is dominant, is $R$ dominant? 
\end{Question}

Although dominance is defined in terms of thick generation in the
singularity category, it is not clear how to descend such generation
along restriction of scalars. We overcome this difficulty by passing
to the homotopy category of acyclic complexes of injective modules,
whose subcategory of compact objects is equivalent to the idempotent
completion of the singularity category by the work of Krause
\cite{Krause2005}. In this setting, restriction of scalars preserves
coproducts, and the stabilization of the residue field can be
controlled. Our main result, \Cref{T1}, answers Takahashi's question
affirmatively.

\begin{theorem}\label{T1}
(See \ref{extension}) Let $R\longrightarrow S$ be a flat local homomorphism between commutative noetherian local rings. If $S$ is dominant, then $R$ is dominant. 
\end{theorem}

As an application, we prove that every Cohen--Macaulay local ring
of minimal multiplicity is dominant, extending Takahashi's result
by removing the infinite residue field assumption; see \Cref{no-infinite}.

The first version of this article also established several special cases of Takahashi's conjecture that every Cohen--Macaulay local ring of finite CM type is dominant \cite[Conjecture 9.1]{Takahashi:2023}. Subsequently, Mifune and Tanigawa \cite{Mifune-Tanigawa} proved that every such ring is uniformly dominant, thereby resolving the conjecture in full. In view of their result, we have omitted the corresponding special
cases from the present version and focused on the descent question.

\section{Notation, terminology, and preliminaries}
Throughout, $R$ denotes a commutative noetherian ring, and all rings considered in this article are assumed to be commutative noetherian. Let $\Mod(R)$ denote the category of all $R$-modules, and let $\mo(R)$ denote its full subcategory consisting of finitely generated $R$-modules. For each $R$-module $M$, we write $\dim(M)$ for the Krull dimension of $M$.
\begin{chunk}
\textbf{Derived category.}
Let $\D(R)$ denote the derived category of $R$-modules, which is a triangulated category with shift functor $\Sigma$. For a complex $X\in \D(R)$, the shift $\Sigma(X)$ is defined by $\Sigma(X)_i=X_{i-1}$ and $\del_{\Sigma(X)}=-\del_X$.

Let $\D^f(R)$ (resp. $\D^b(R)$) denote the full subcategory of $\D(R)$ consisting of complexes $X$ such that $\h_i(X)$ is finitely generated over $R$ for every $i\in\mathbb Z$ (resp. $\h_i(X)=0$ for $|i|\gg0$), where $\h_i(X)$ is the $i$-th homology of $X$. Set $$\D^f_b(R)\colonequals \D^f(R)\cap\D_b(R).$$  Note that a complex over $R$ is in $\D^f_b(R)$ if and only if the total homology $\bigoplus_{i\in\Z}\h_i(X)$ is finitely generated over $R$. 
\end{chunk}
\begin{chunk}
\textbf{Singularity category.}
An object $X\in\D^f_b(R)$ is called \emph{perfect} if it is isomorphic in $\D(R)$ to a bounded complex of finitely generated projective $R$-modules. We write $\per(R)$ for the full subcategory of $\D^f_b(R)$ consisting of perfect complexes. This is a triangulated subcategory of $\D^f_b(R)$. The \emph{singularity category} of $R$ is the Verdier quotient
\[
\D_{\sg}(R)\colonequals \D^f_b(R)/\per(R).
\]
This construction was introduced by Buchweitz \cite[Definition 1.2.2]{Buchweitz}; see also \cite{Orlov}.
Note that $\D_{\sg}(R)=0$ if and only if every finitely generated $R$-module has finite projective dimension. 
\end{chunk}

\begin{chunk}
\textbf{Thick subcategory.}
Let $\T$ be a triangulated category. A full subcategory $\C$ of $\T$ is called \emph{thick} if $\C$ is closed under shifts, cones, and direct summands. 
For each object $X$ in $\T$, let $\thick_\T(X)$ denote the smallest thick subcategory of $\T$ containing $X$. This can be realized as the intersection of all thick subcategories of $\T$ containing $X$. 
For example, $\per(R)$ is a thick subcategory of $\D^f_b(R)$; see \cite[Lemma 1.2.1]{Buchweitz}. Moreover, $\per(R)=\thick_{\D^f_b(R)}(R)$.

Following \cite[2.2.4]{ABIM} and \cite[Section 2]{BonVD}, the subcategory $\thick_\T(X)$ admits the following inductive description. Define
$\thick_\T^0(X)=\{0\},
$
and let $\thick_\T^1(X)$ be the smallest full subcategory of $\T$ that contains $X$ and is closed under shifts, finite direct sums, and  direct summands. For $n\geq 2$, define $\thick_\T^n(X)$ to be the full subcategory consisting of those objects $Y\in\T$ for which there exists an exact triangle
\[
Y_1\longrightarrow Y\oplus Y'\longrightarrow Y_2\longrightarrow \Sigma (Y_1)
\]
in $\T$
with
$
Y_1\in\thick_\T^1(X)$ and $
Y_2\in\thick_\T^{n-1}(X)
$.
Then
\[
\thick_\T(X)=\bigcup_{n\geq 0}\thick_\T^n(X).
\]
\end{chunk}

\begin{chunk}
    \textbf{Localizing subcategory.}
    Let $\T$ be a triangulated category admitting arbitrary coproducts. 
A full subcategory of a triangulated category $\T$ is called \emph{localizing} if it is a triangulated subcategory closed under arbitrary coproducts. Given a class of objects $S$ in $\T$, the \emph{localizing subcategory generated by $S$}, denoted by $\operatorname{Loc}_{\T}(S)$, is the smallest localizing subcategory of $\T$ containing $S$. Note that any localizing subcategory of $\T$ is thick.
\end{chunk}
 
 \begin{chunk}\label{compactly gen}
     \textbf{Compactly generated triangulated categories.}
Let $\T$ be a triangulated category admitting arbitrary coproducts. An object $X$ of $\T$ is called \emph{compact} if the functor $\Hom_{\T}(X,-)$ preserves arbitrary coproducts. Equivalently, for every family of objects $Y_i(i\in I)$ in $\T$, the canonical map
$$
\bigoplus_{i\in I}\Hom_{\T}(X,Y_i)\longrightarrow  \Hom_{\T}(X,\bigoplus_{i\in I}Y_i)
$$
is an isomorphism. Denote by $\T^{c}$ the full subcategory of $\T$ consisting of all compact objects. 

The triangulated category $\T$ is said to be \emph{compactly generated} if there exists a set of compact objects $S$ such that
$
\operatorname{Loc}_\T(S)=\T.
$
In this situation, the compact objects are precisely those lying in the thick subcategory generated by $S$; that is,
$
\T^{c}=\thick_{\T}(S),
$ 
see \cite[Lemma 2.2]{Neeman92}. For example, the derived category of $R$ is compactly generated by $R$. 
 \end{chunk}

\begin{chunk}\label{Kac}
\textbf{Homotopy category of (acyclic) complexes of injective modules.} Let $\K(\Inj R)$ denote the homotopy category of complexes of injective $R$-modules. Set
$$
\K_{\ac}(\Inj R)=\{X\in\K(\Inj R)\mid \h_i(X)=0 \text{ for all }i\in\Z\}.
$$
 This is a localizing subcategory of $\K(\Inj R)$. Krause \cite[Proposition 2.3]{Krause2005} proved that $\K(\Inj R)$ is compactly generated and there is a triangle equivalence
 $$
 \text{i}\colon \D^{f}_b(R)\xrightarrow \equiv\K(\Inj R)^{c}
 $$
 induced by taking injective resolutions.
\end{chunk}

For a local ring $(R,\m,k)$, we use $\embdim(R)$ to denote the \emph{embedding dimension} of $R$, given by
$\embdim(R)\colonequals\rank_k(\m/\m^2),
$
which is the minimal number of generators of $\m$. The \emph{codimension} of $R$ is the integer
\[
\codim(R)\colonequals\rank_k(\m/\m^2)-\dim(R).
\]

\begin{chunk}
\textbf{Hilbert--Samuel multiplicity.}
Let $(R,\m)$ be a local ring. For each finitely generated $R$-module $M$, the \emph{Hilbert--Samuel multiplicity} of $M$ with respect to $\m$ is defined by
\[
{\rm e}(M)\colonequals \lim_{n\to\infty} d!\cdot\frac{\ell(M/\m^nM)}{n^d},
\]
where $d=\dim(M)$ and $\ell(-)$ represents the length of the $R$-module. See \cite[Section 4.6]{BH} for details.


Assume $(R,\m,k)$ is a Cohen--Macaulay local ring; that is,
the length of a maximal $R$-regular sequence in $\m$ is equal to $\dim(R)$.  In \cite{Abhyankar}, Abhyankar proved that the Hilbert--Samuel  multiplicity of $R$
satisfies the inequality
\[
{\rm e}(R)\geq \embdim(R)-\dim(R)+1.
\]
$R$ is said to have \emph{minimal multiplicity} if equality holds
in the above inequality.
If the residue field $k$ is infinite, then $R$ has minimal multiplicity if and
only if
$
\m^2=Q\m
$
for some minimal reduction $Q$ of $\m$; see 
\cite[Exercise 4.6.14]{BH}.
\end{chunk}

\begin{chunk}\label{dominant}
\textbf{(Uniform) dominant local ring.}
Let $(R,\m,k)$ be a local ring. $R$ is called \emph{dominant} if
\[
k\in \thick_{\D_{\sg}(R)}(X)
\]
for every nonzero complex $X\in\D_{\sg}(R)$. See details in \cite{Takahashi:2023}. Recall that $k\in\thick_{\D_{\sg}(R)}(X)$ if and only if
$k\in\thick_{\D^f_b(R)}(X\oplus R)$.

Note that every object in $\D_{\sg}(R)$ is isomorphic to $\Sigma^nM$ for some finitely generated $R$-module $M$ and some $n\in\mathbb Z$. Therefore, the local ring $(R,\mathfrak m,k)$ is dominant if and only if
$k\in \thick_{\D_{\sg}(R)}(M)
$
for every nonzero finitely generated $R$-module $M$ of infinite projective dimension.

In \cite[Definition 5.3]{Takahashi:2026}, Takahashi introduced the \emph{dominant index} of $R$, defined by
\[
{\rm dx}(R)\colonequals \inf\{ n\in\mathbb{Z}_{\geq -1}\mid 
k\in\thick^{n+1}_{\D_{\sg}(R)}(X) \text{ for every }0\neq X\in\D_{\sg}(R)\}.
\]
A local ring $R$ is said to be \emph{uniformly dominant} if ${\rm dx}(R)<\infty$. Clearly, every uniformly dominant local ring is dominant.

\end{chunk}
\begin{example}\label{example}
    By \cite[Proposition 5.10]{Takahashi:2023}, known examples of dominant local rings include hypersurfaces, Cohen--Macaulay local rings with infinite residue fields and minimal multiplicity, local rings with quasi-decomposable maximal ideals \cite{NT}, and Burch rings \cite{DKT2020}.
\end{example}

\section{Descent of dominance}
The main result is \Cref{T1} from the introduction; see \Cref{extension}. It shows that the following question of Takahashi has an affirmative answer.
\begin{chunk}\label{flathom}
   In \cite[Question 7.5]{Takahashi:2023},  Takahashi asked:
   \begin{Question}
       Let $(R,\m)\rightarrow (S,\n)$ be a flat local homomorphism. If $S$ is dominant, is $R$ dominant?
   \end{Question}
\end{chunk}

\begin{chunk}\label{Letz}
     Let $(R,\m)\rightarrow (S,\n)$ be a flat local homomorphism. Let $X,G$ be objects in $\D^f_b(R)$. For each $n\geq 0$, Letz \cite[Corollary 2.11]{Letz} observed that
     $$
     X\in \thick^n_{\D^f_b(R)}(G)\iff S\otimes_R X\in \thick^n_{\D^f_b(S)}(S\otimes_R G).
     $$
     In particular, $X\in\thick_{\D^f_b(R)}(G)$ if and only if $S\otimes_R X\in \thick_{\D^f_b(S)}(S\otimes_R G)$. Taking $G=R$, we obtain
     $$
     X=0 \text{ in }\D_{\sg}(R) \iff S\otimes_R X =0\text{ in } \D_{\sg}(S).
     $$
\end{chunk}
For a commutative noetherian ring $A$, recall that a complex $J$ of injective $A$-modules is called \emph{$K$-injective} if the Hom complex
$
\Hom_A(E,J)
$
is acyclic for every acyclic complex $E$ of $A$-modules. 

\begin{chunk}\label{recollement}
Let $A$ be a commutative noetherian ring. By \cite[Corollary 4.3]{Krause2005}, there is a recollement
\[
\begin{tikzcd}[column sep=large, row sep=large]
\K_{\ac}(\Inj A)
  \arrow[r, "I_A"]
&
\K(\Inj A)
  \arrow[l, shift left=2.2ex, "I_{A,\rho}"]
  \arrow[l, shift right=2.2ex, "I_{A, \lambda}"']
  \arrow[r, "Q_A"]
&
\D(A)
  \arrow[l, shift left=2.2ex, "Q_{A,\rho}"]
  \arrow[l, shift right=2.2ex, "Q_{A,\lambda}"']
\end{tikzcd}
\]
in the sense of \cite{BBD}, where $I_A$ and $Q_A$ are canonical functors. Note that the right adjoint $Q_{A,\rho}$ of $Q_A$ is induced by taking $K$-injective resolutions. 
Let 
$$
\sigma_A=I_{A,\lambda} Q_{A,\rho}\colon \D(A)\rightarrow \K_{\ac}(\Inj A)
$$
 denote the \emph{stable localization} functor. This is introduced in 
 \cite[Section 5]{Krause2005}. By \cite[Corollary 5.5]{Krause2005}, the composition of the following functors
 $$
 \Mod(A)\to \D(A)\xrightarrow{\sigma_A} \K_{\ac}(\Inj A)
 $$
 preserves all coproducts and $\sigma_A(A)=0$. Moreover, $\sigma_A$ induces a fully faithful functor
 $$
 \overline{\sigma_A}\colon \D_{\sg}(A)\rightarrow \K_{\ac}(\Inj A)^c
 $$
 which is an equivalence up to direct summands; see \cite[Corollary 5.4]{Krause2005}. 

 For each $X\in\K(\Inj A)$, there exists an exact triangle
 \begin{equation}\label{decomposition}
      Q_{A,\lambda}Q_A X\to X\to I_AI_{A,\lambda}X\to \Sigma Q_{A,\lambda}Q_A X;
 \end{equation}
 see \cite[page 1135]{Krause2005}. 
 Let \(\mathcal L_A\) denote the left orthogonal of \(\K_{\ac}(\Inj A)\) in
\(\mathrm K(\operatorname{Inj}A)\). Namely, for each object $X\in\K(\Inj A)$, it is in $\mathcal L_A$ if and only if 
 $\Hom_{\K(\Inj A)}(X,Y)=0$ for all $Y\in \K_{\ac}(\Inj A)$. 
Note that $Q_{A,\lambda}Q_A X\in \mathcal L_A$ and $I_AI_{A,\lambda}X\in\K_{\ac}(\Inj A)$. Moreover, any exact triangle in $\K(\Inj A)$
$$
X_1\to X\to X_2\to \Sigma X_1
$$
with $X_1\in \mathcal L_A$ and $X_2\in \K_{\ac}(\Inj A)$ is isomorphic to the exact triangle (\ref{decomposition}).
\end{chunk}
\begin{lemma}\label{restric-fun}
Let \(R\to S\) be a flat homomorphism. The restriction of scalars induces a coproduct-preserving triangle functor
\[
U\colon\K_{\ac}(\Inj S)\longrightarrow\K_{\ac}(\Inj R).
\]
Moreover, restriction of every \(K\)-injective complex of \(S\)-modules is \(K\)-injective over \(R\).
\end{lemma}
\begin{proof}
Assume \(J\) is an injective \(S\)-module. For each \(R\)-module \(N\), there is a natural isomorphism
\[
\operatorname{Hom}_R(N,J)
\cong
\operatorname{Hom}_S(S\otimes_RN,J).
\]
Since $S$ is flat over $R$ and $J$ is injective over $S$, the functor $\Hom_R(-,J)$ is exact, and hence \(J\) is injective as an \(R\)-module.
Since the restriction functor is exact, it takes acyclic complexes to acyclic complexes and defines the desired functor. Note that direct sums of injective modules are injective over noetherian rings. Coproducts in the relevant homotopy categories are therefore represented termwise (see \cite[Lemma 1.1]{BN}), and restriction preserves them.

Next, we prove the second statement. Let \(J\) be a \(K\)-injective complex over \(S\), and let \(E\) be an acyclic complex of \(R\)-modules. There is a natural isomorphism of Hom complexes
\[
\operatorname{Hom}_R(E,J)
\cong
\operatorname{Hom}_S(S\otimes_RE,J).
\]
Flatness makes \(S\otimes_RE\) acyclic.  Since \(J\) is \(K\)-injective over $S$, the complex on the right is acyclic. Thus, \(J\) is \(K\)-injective over \(R\).
\end{proof}

\begin{remark}
 Keep the assumption of \Cref{restric-fun}. Let $X$ be a complex of $S$-modules. We still denote it by $X$ when viewing it as a complex of $R$-modules. For $X\in \K_{\ac}(\Inj S)$, we use $U(X)$ to denote the same complex regarded as a complex of $R$-modules, which belongs to $\K_{\ac}(\Inj R)$. This notation is introduced only to make the proof of \Cref{extension} clearer and should not cause any confusion.
\end{remark}

\begin{proposition}\label{res-f}
Let \(R\to S\) be a flat homomorphism and suppose that
$\operatorname{pd}_R S<\infty. 
$
Then

(1) For each $X\in\D(S)$,  
there is an isomorphism
$
U\sigma_S(X)\cong\sigma_R(X)
$
in $\K_{\ac}(\Inj R)$. 

(2)
If
$
(R,\mathfrak m,k)\longrightarrow(S,\mathfrak n,l)
$
is local, then 
$
\sigma_R(k)
$
is a direct summand of $U\sigma_S(l)$ 
in \(\K_{\ac}(\Inj R)\).
\end{proposition}
\begin{proof}
(1) Let \(E\) be an acyclic complex of injective \(R\)-modules. Since \[
\operatorname{Hom}_S(-,\operatorname{Hom}_R(S,E_i))
\cong
\operatorname{Hom}_R(-,E_i)
\]
for each $i\in\mathbb Z$,
we get that the complex
$
\operatorname{Hom}_R(S,E)
$
is a complex of injective modules over \(S\).
Let $K_i$ denote the kernel of the map $\del_i^E\colon E_i\to E_{i-1}$. Since $E$ is acyclic, there is a short exact sequence
\[
0\longrightarrow K_{i+1}\longrightarrow E_{i+1}
\longrightarrow K_i\longrightarrow 0
\]
for each $i\in\mathbb Z$. As each $E_{i+1}$ is injective,
dimension shifting yields
\[
\Ext_R^1(S,K_i)\cong\Ext_R^{n+1}(S,K_{i+n})
\]
for each $n\geq0$. Since $\pd_R S<\infty$, we obtain
$\Ext_R^1(S,K_i)=0$ for  each $i\in\mathbb Z$.
Consequently, $\Hom_R(S,E)$ is an acyclic complex of
injective $S$-modules.

 Next, keep the notation as in \Cref{recollement}. If \(L\in\mathcal L_S\) and
\(E\in\K_{\ac}(\Inj R)\), there is an adjunction isomorphism
\[
\operatorname{Hom}_{\mathrm K(\operatorname{Inj}R)}(L,E)
\cong
\operatorname{Hom}_{\mathrm K(\operatorname{Inj}S)}
\bigl(L,\operatorname{Hom}_R(S,E)\bigr).
\]
This is zero as we have $\Hom_R(S,E)\in \K_{\ac}(\Inj S)$. 
Thus, restriction of scalars sends every object of
$\mathcal L_S$ to an object of $\mathcal L_R$.

For each \(X\in\mathrm D(S)\), we have $Q_{S,\rho}(X)\in \K(\Inj S)$,  and then the recollement in \Cref{recollement} induces an exact triangle
\[
L_X\longrightarrow Q_{S,\rho}(X)
\longrightarrow I_S I_{S,\lambda} Q_{S,\rho}(X)
\longrightarrow\Sigma L_X
\]
in $\K(\Inj S)$,
where \(L_X\in\mathcal L_S\); see \Cref{recollement}. Note that $I_S I_{S,\lambda} Q_{S,\rho}(X)=\sigma_S(X)$ is acyclic. After restriction, the first term is in \(\mathcal L_R\), the middle term is in $\K(\Inj R)$ by the proof of \Cref{restric-fun}, and the third one is in $\K_{\ac}(\Inj R)$ by \Cref{restric-fun}. The uniqueness of the corresponding orthogonal decomposition in \Cref{recollement} implies
\[
U\sigma_S(X)\cong I_{R,\lambda}Q_{S,\rho}X.
\]
Since $Q_{S,\rho}X$ is $K$-injective over $S$ (see \Cref{recollement}), it follows from \Cref{restric-fun} that it is also $K$-injective over $R$ . Thus, $Q_{S,\rho}X\cong Q_{R,\rho}X$  in $\K(\Inj R)$, and hence we conclude that $U\sigma_S(X)\cong \sigma_R(X)$.

(2) The $R$-module $l$ is also a $k$-vector space. Thus, $k$ is a direct summand of $l$ over $k$.  This also yields that $k$ is a direct summand of $l$ as $R$-modules. Applying the functor \(\sigma_R\), we get that the object $\sigma_R(k)$ is a direct summand of $\sigma_R(l)$. 
Combining this with $\sigma_R(l)\cong U\sigma_S(l),
$
we conclude that $\sigma_R(k)$ is a direct summand of $U\sigma_S(l)$ in $\K_{\ac}(\Inj R)$. 
\end{proof}

\begin{proposition}\label{res:contain}
Let \(R\to S\) be a flat homomorphism with \(\operatorname{pd}_RS<\infty\), and let \(M\) be a finitely generated \(R\)-module. Then
$
\sigma_R(S\otimes_RM)
\in
\operatorname{Loc}_{\K_{\ac}(\Inj R)}(\sigma_R(M)) 
$ and we have
\[
U\sigma_S(S\otimes_RM)
\in
\operatorname{Loc}_{\K_{\ac}(\Inj R)}(\sigma_R(M)).
\]
\end{proposition}
\begin{proof}
Put \(d=\operatorname{pd}_RS\), and choose a projective resolution
\[
0\longrightarrow P_d\longrightarrow\cdots
\longrightarrow P_0\longrightarrow S\longrightarrow0.
\]
Because \(S\) is flat, every syzygy in this resolution is flat.  Thus, tensoring this resolution with $M$ yields an exact sequence
\begin{equation}\label{equation}
    0\longrightarrow P_d\otimes_RM\longrightarrow\cdots
\longrightarrow P_0\otimes_RM
\longrightarrow S\otimes_RM\longrightarrow0.
\end{equation}

For every set \(I\), it follows from \Cref{recollement} that
$
\sigma_R(M^{(I)})\cong\sigma_R(M)^{(I)}.
$
Since each \(P_i\) is a direct summand of a free module \(R^{(I_i)}\), the module \(P_i\otimes_RM\) is a direct summand of \(M^{(I_i)}\). Combining with $
\sigma_R(M^{(I)})\cong\sigma_R(M)^{(I)}
$, we get
\begin{equation}\label{123}
   \sigma_R(P_i\otimes_RM)
\in
\operatorname{Loc}_{\K_{\ac}(\Inj R)}(\sigma_R(M))
\end{equation}
for each $0\leq i\leq d$. 
By (\ref{equation}), $S\otimes_R M\in \thick_{\D(R)}(\oplus_{i=0}^d P_i\otimes _R M)$, and hence 
\begin{equation}\label{234}
    \sigma_R(S\otimes_R M)
\in \thick_{\K_{\ac}(\Inj R)}(\oplus_{i=0}^d \sigma_R(P_i\otimes _R M));
\end{equation}
see \cite[Lemma 2.4(6)]{ABIM}.
By (\ref{123}) and (\ref{234}), we have 
\[
\sigma_R(S\otimes_RM)
\in
\operatorname{Loc}_{\K_{\ac}(\Inj R)}(\sigma_R(M)).
\]
The second statement follows by combining this with \Cref{res-f}.
\end{proof}

\begin{theorem}\label{extension}
Let
$
(R,\mathfrak m,k)\longrightarrow(S,\mathfrak n,l)
$
be a flat local homomorphism.
If \(S\) is dominant, then \(R\) is dominant. 
\end{theorem}
\begin{proof}
By \cite[Theorem 3.2.6]{RG} and \cite[Proposition 6]{Jensen}, every flat module has finite projective dimension over a commutative noetherian ring of finite Krull dimension; see also \cite[5.9]{DKLO}. Since $R$ is a noetherian local ring, it has finite Krull
dimension, and therefore $\pd_R(S)<\infty$. Hence, the assumptions in \Cref{res-f} and \Cref{res:contain} are satisfied.

Fix a finitely generated \(R\)-module \(M\) of infinite projective dimension. By \Cref{Letz}, we have
$
\operatorname{pd}_S(S\otimes_RM)=\infty.
$
Together with the dominance of $S$, this yields
\[
l\in
\operatorname{thick}_{\D^f_b(S)}
(S\oplus (S\otimes_RM)).
\]
Since $\sigma_S(S)=0$ (see \Cref{recollement}), we obtain
$
\sigma_S(l)\in
\operatorname{thick}_{\K_{\ac}(\Inj S)}
\bigl(\sigma_S(S\otimes_RM)\bigr).
$
After restriction, \Cref{res:contain} implies that
$
U\sigma_S(l)\in
\operatorname{Loc}_{\K_{\ac}(\Inj R)}(\sigma_R(M)).
$
By \Cref{res-f}, the compact object \(\sigma_R(k)\) is a direct summand of \(U\sigma_S(l)\). Therefore,
$$
\sigma_R(k) \in
\operatorname{Loc}_{\K_{\ac}(\Inj R)}(\sigma_R(M)).
$$

Note that the category \(\K_{\ac}(\Inj R)\) is compactly generated by \cite[Corollary 5.4]{Krause2005}, and the stabilizations of finitely generated modules are compact (see \Cref{recollement}). Hence, we obtain the first inclusion below:
\begin{align*}
   \sigma_R(k) &\in
\K_{\ac}(\Inj R)^c\cap \operatorname{Loc}_{\K_{\ac}(\Inj R)}
 (\sigma_R(M))\\
 & \subseteq \operatorname{Loc}_{\K_{\ac}(\Inj R)}
 (\sigma_R(M))^c\\
& =\thick_{\K_{\ac}(\Inj R)^c}(\sigma_R(M)), 
\end{align*}
where the equality follows from \Cref{compactly gen}, since $\operatorname{Loc}_{\K_{\ac}(\Inj R)}
 (\sigma_R(M))$ is compactly generated by $\sigma_R(M)$ and $\K_{\ac}(\Inj R)^c$ is a thick subcategory of $\K_{\ac}(\Inj R)$.

Combining this with the fully faithful functor $\D_{\sg}(R)\to \K_{\ac}(\Inj R)$ induced by $\sigma_R$, we conclude from \cite[Proposition 3.15]{DLL} that
\[
k\in
\operatorname{thick}_{\D_{\sg}(R)}(M).
\]
This proves that $R$ is dominant, since $M$ was chosen as an arbitrary finitely generated $R$-module of infinite projective dimension.
\end{proof}

\begin{remark}\label{completion}
       For a local ring $(R,\m)$, Takahashi \cite[Corollary 5.8]{Takahashi:2023} proved that $R$ is dominant if and only if $\widehat{R}$ is dominant. Note that the descent direction is also a direct consequence of \Cref{extension}. 
\end{remark}

The following result was proved by Takahashi \cite[Proposition 5.10]{Takahashi:2023} under the assumption that the residue field is infinite. By \Cref{extension}, this assumption can be removed.

\begin{corollary}\label{no-infinite}
Every Cohen--Macaulay local ring of minimal multiplicity is dominant.
\end{corollary}

\begin{proof}
Let $(R,\mathfrak m,k)$ be a Cohen--Macaulay local ring of minimal
multiplicity, and set
$
S\colonequals R[t]_{\mathfrak mR[t]}.
$
Then $S$ is a local ring with maximal ideal $\mathfrak mS$, and its residue field
is the infinite field $k(t)$. Note that the natural map $R\to S$ is a faithfully flat local homomorphism. 
 By
\cite[Theorem A.11]{BH}, we have $\dim(R)=\dim(S)$. Moreover, $S$ is
Cohen--Macaulay by \cite[Theorem 2.1.7]{BH}.

For each $n\geq 0$, there are isomorphisms
\begin{align*}
(\mathfrak mS)^n/(\mathfrak mS)^{n+1}
&\cong \mathfrak m^nS/\mathfrak m^{n+1}S\\
&\cong (\mathfrak m^n/\mathfrak m^{n+1})\otimes_R S\\
&\cong (\mathfrak m^n/\mathfrak m^{n+1})\otimes_k k(t),
\end{align*}
where the second isomorphism follows from the flatness of $S$ over $R$.
It follows that
$
\embdim(S)=\embdim(R)
$
and
$
{\rm e}(S)={\rm e}(R).
$
Together with $\dim(R)=\dim(S)$ and the minimal multiplicity of $R$, these
equalities show that $S$ has minimal multiplicity. Hence, 
\cite[Proposition~5.10]{Takahashi:2023} yields that $S$ is dominant, and
therefore $R$ is dominant by \Cref{extension}.
\end{proof}

\begin{remark}
By \cite[Proposition~6.7(1d) and Theorem~6.14(2)]{KT2026},
a non-Gorenstein Cohen--Macaulay local ring $(R,\m)$ with infinite
residue field is uniformly dominant if either ${\rm e}(R)\leq 5$
or $\codim R\leq 2$.
The same argument as in \Cref{no-infinite} shows that $R$ is dominant in these
two cases without any restriction on the residue field.

Indeed, set
$S=R[t]_{\mathfrak mR[t]}$. Then $S$ is Cohen--Macaulay with infinite
residue field and has the same dimension, embedding dimension, and
multiplicity as $R$. Moreover, $S$ is non-Gorenstein, since
Gorensteinness descends along flat local homomorphisms; see \cite[Corollary 3.3.15]{BH}.
By the results cited above, $S$ is dominant.
Applying \Cref{extension}, we conclude that $R$ is dominant.
\end{remark}

\begin{ack}
  This work was carried out by the author with the assistance of Eureka, a multi-agent system developed by JIUCHONG at the University of Science and Technology of China for mathematical research. The author would like to thank the JIUCHONG team for providing access to Eureka. The author was supported by the National Natural Science Foundation of China (No. 12401046).
\end{ack}

\bibliographystyle{amsplain}
\bibliography{ref}
\end{document}